\documentclass[12pt]{article}
\usepackage{amsmath,amssymb,amsthm}
\usepackage[margin=3cm]{geometry}
\usepackage{blindtext}
\usepackage{titlesec}
\title{Sections and Chapters}

\title{\bf $L_{1}$- Properties of vector-valued Banach algebras}
\author{Maryam Aghakoochaki  $^1$, \thanks{2020 Mathematics Subject Classifcation. Primary: 46J05; Secondary: 46J10} ,  \and Ali Rejali $^2$, \thanks{Corresponding author}}

\date{
	$^1$Isfahan University  \\ \texttt{mkoochaki@sci.ui.ac.ir, Orcid: 0000-0002-3851-6550}\\%
	$^2$Department of Pure Mathematics, Faculty of Mathematics and Statistics, University of Isfahan, Isfahan 81746-73441, Iran \\ \texttt{rejali@sci.ui.ac.ir, Orcid: 0000-0001-7270-665X}\\[2ex]% 
\today
}

\newcommand{\tnorm}[1]{{\left\vert\kern-0.25ex\left\vert\kern-0.25ex\left\vert #1 
    \right\vert\kern-0.25ex\right\vert\kern-0.25ex\right\vert}}

\theoremstyle{plain}
\newtheorem{thm}{Theorem}[section]

\newtheorem{cor}[thm]{Corollary}
\newtheorem{lem}[thm]{Lemma}
\theoremstyle{definition}
\newtheorem{rem}[thm]{\bf Remark}
\newtheorem{ex}[thm]{\bf Example}

\newtheorem{DEF}[thm]{\bf Definition}

\begin{document}
\maketitle

\begin{abstract}
		Let   $G$ be a locally compact group and $A$ be a commutative semisimple Banach algebra over the scalar field $\mathbb{C}$.
The correlation between different types of  $\textup{BSE}$- Banach algebras $A$, and the Banach algebras $L^{1}(G, A)$  are assessed. It is found and approved that 
$M(G, A) = L^{1}(G, A)$  if and only if $G$ is discrete. Furthermore, some properties of vector-valued measure algebras on groups are given, so that $M(G, A)$ is a 
convolution measure algebra.
		
\end{abstract}

%\tableofcontents
\section{Introduction}
 In this paper, $A$ is a commutative Banach algebra with the dual space $A^{*}$ and $G$ be a locally compact group.
Let  $\Delta(A)$ be the character space of $ A$ with the Gelfand topology. In fact,  $\Delta(A)$ is the set consisting of all non-zero multiplicative
linear functionals on $A$.  We denote by $C_{b}(\Delta(A))$ the Banach algebra of all complex-valued continuous and bounded functions
on $\Delta(A)$, with sup-norm $\|.\|_{\infty}$ and pointwise product. The Gelfand map $A\to C_{b}(\Delta(A))$ is a  continuous algebra homomorphism on  $A$, where $a\mapsto \hat a$ such that $\hat{a}(\varphi)= \varphi(a)$; ($\varphi\in \Delta(A)$).
  If the Banach algebra $A$ is  semisimple, then the Gelfand map $\Gamma_{A}: A\to\widehat A$, $a\mapsto \hat a$, is injective.
Let $A$ be a semisimple Banach algebra. Then $\Delta(A)$ is compact if and only if $A$ is unital; See \cite{kan}.
A continuous linear operator $T$ on $A$ is named a multiplier if for all $x,y\in A$, $T(xy)=xT(y)$; See \cite{kan}.
The set of all multipliers on $A$ will be denoted by $M(A)$. It is obvious that $M(A)$ is a Banach algebra, and if $A$ is an unital Banach algebra, then 
$M(A)\cong A$.  As observed in   \cite{klar},  for each $T\in M(A)$ there exists a unique bounded  continuous function $\widehat{T}$ on $\Delta(A)$ expressed as:
$$\varphi(Tx)=\widehat{T}(\varphi)\varphi(x),$$
for all $x\in A$ and $\varphi\in \Delta(A)$. By setting 
$$
M(A\widehat)= \{\hat T: ~ T\in M(A)\}
$$
and 
$$
{\cal M}(A):= \{\sigma\in C_{b}(\Delta(A): ~ \sigma. \hat A\subseteq \hat A\}.
$$
If $A$ is semisimple, then ${\cal M}(A)= M(A\widehat)$.
A bounded complex-valued continuous function $\sigma$ on $\Delta(A)$  is named a $\textup{BSE}$- function, if there exists a positive real number $\beta$ in a sense that for every finite complex number $c_{1},\cdots,c_{n}$,  and  the same many $\varphi_{1},\cdots,\varphi_{n}$ in $\Delta(A)$ the following inequality
$$\mid\sum_{i=1}^{n}c_{i}\sigma(\varphi_{i})\mid\leq \beta\|\sum_{i=1}^{n}c_{i}\varphi_{i}\|_{A^{*}}$$
holds.\\

The set of all $\textup{BSE}$- functions is expressed by $C_{\textup{BSE}}(\Delta( A))$, where for 
 each $\sigma$, the $\textup{BSE}$- norm of $\sigma$, $\|\sigma\|_{\textup{BSE}}$ 
is the infimum of all  $\beta$s  applied in the above inequality.   According to [Lemma1, \cite{E6}], $(C_{\textup{BSE}}(\Delta(A)), \|.\|_{\textup{BSE}})$ is a semisimple Banach subalgebra of $C_{b}(\Delta(A))$. The algebra $A$ is named a $\textup{BSE}$- algebra  if it meets the following equality holds:
$$ C_{\textup{BSE}}(\Delta(A)) = \widehat{M(A)}.$$ 
If $A$ is unital, then $\widehat{M(A)}= \widehat{A}\mid_{\Delta(A)}$, indicating that $A$ is a $\textup{BSE}$- algebra if and only if  $C_{\textup{BSE}}(\Delta(A)) = \widehat{A}\mid_{\Delta(A)}$. In this paper, by $\widehat A$ we means $ \widehat{A}\mid_{\Delta(A)}$.
The semisimple Banach algebra $A$ is named a norm- $\textup{BSE}$ algebra if there exists some  $K>0$ in a sense that for each $a\in A$, the following holds: 
$$
\|a\|_{A}\leq K\|\hat{a}\|_{\textup{BSE}}.
$$
See\cite{DU}.

 The notion of $\textup{BSE}$- algebras was introduced and assessed by Takahasi and Hatori in \cite{E6}.
Later, they explained BSE properties; See  \cite{E7} and \cite{E8}.
 Next, several authors have studied this concept for various kinds of Banach algebras; See \cite{AK}, \cite{AK2}, \cite{SB}, \cite{N2}, \cite{K1}. 
For more details see \cite{DU}, \cite{DU2}. In 2007, Inoue and Takahasi in \cite{F1}, introduced the concept of $\textup{BED}$- concerning Doss, where Fourier-Stieltjes transforms of absolutely continuous measures are specified. 
The Bochner-Eberlein-Doss  ($\textup{BED}$) is derived from the famous
theorem proved in \cite{RD1} and \cite{RD2}. They proved that if $G$ is a locally compact Abelian
group, then the group algebra $L_{1}(G)$  is a $\textup{BED}$- algebra. Later, the researcher in  \cite{AP},  showed that $l^{p}(X, A)$ is a $\textup{BED}$- algebra if and only if $A$ is so.
The researchers in \cite{ARC},  studied ${\textup{BSE}}$- properties of  $L^{p}(G, A)$. Recently, the authors studied different types of ${\textup{BSE}}$- algebras of $C_{0}(X, A)$ in \cite{MAR}.
The
interested reader is referred to \cite{F1}, \cite{kan}, and \cite{klar}.
In general, 
\begin{align*}
\|\widehat a\|_{\infty} &\leq \|\widehat a\|_{\textup{BSE}} \leq \|\widehat a\|_{A^{**}}= \|a\|_{A}      
\end{align*}
for each $a\in A$.
The function $\sigma\in C_{BSE}(\Delta(A))$ is  a ${\textup{BED}}$- function, if for all $\epsilon>0$, there exists some compact set such $K\subseteq \Delta(A)$, where  for all $c_{i}\in\mathbb C$ and for all $\varphi_{i}\in\Delta(A)\backslash K$ the following inequality:
$$
|\sum_{i=1}^{n} c_{i}\sigma(\varphi_{i}) |\leq \epsilon\|\sum_{i=1}^{n}c_{i}\varphi_{i}\|_{A^{*}}
$$
holds. This definition of $\textup{BED}$-  functions is a modification of the definition in \cite{RD1}.
  The set of all $\textup{BED}$-  functions is expressed by the $C_{BSE}^{0}(\Delta(A))$.  Clearly, $C_{BSE}^{0}(\Delta(A))$ is a closed ideal of $ C_{BSE}(\Delta(A))$; see[\cite{F1}, Corollary 3.9].
In the following, we define several types of $\textup{BSE}$- algebras that are used in this paper; See\cite{MAR}.
\begin{DEF}
Let $A$ be a commutative  Banach algebra where $\Delta(A)$ is non-empty. Then:\\
(i) $A$ is a $\textup{BSE}$- algebra if and only if $ C_{\textup{BSE}}(\Delta(A)) = \widehat{M(A)}$.\\
(ii) $A$ is a weak- $\textup{BSE}$ algebra if and only if $C_{\textup{BSE}}(\Delta(A)) = \hat A$.\\
(iii) $A$ is a  $\textup{BED}$- algebra if and only if $C_{\textup{BSE}}^{0}(\Delta(A)) = \hat{A}$.\\
(iv)  $A$ is a weak- $\textup{BED}$- algebra if and only if $C_{\textup{BSE}}^{0}(\Delta(A)) = \widehat{M(A)}$.
\end{DEF}
%^^^^^^^^^^^^^^^^^^^^^^^^^^^^^^^^^^^^^^^^^^^^^^^^^^^^^^^^^^^^^^^^^^^^^^^
In the following, we characterize the different types of   $\textup{BSE}$- algebras for commutative group algebra $A= L^{1}(G)$; See\cite{RD1}.
\begin{ex}
Let $G$ be a locally compact Abelian group and $A= L^{1}(G)$. Then \\
(i) $ C_{\textup{BSE}}(\Delta(A))= \widehat{M(G)}$, so $L^{1}(G)$ is a $\textup{BSE}$- algebra.\\
(ii)  $L^{1}(G)$ is a weak- $\textup{BSE}$  algebra if and only if  $G$ is discrete.\\
(iii) $ C_{\textup{BSE}}^{0}(\Delta(A))= {L^{1}(G\widehat)}$, so $L^{1}(G)$ is a $\textup{BED}$- algebra.\\
(iv) $L^{1}(G)$ is a weak- $\textup{BED}$ algebra if and only if  $G$ is discrete.
 \end{ex}
Let   $L^{1}(G, A)$ be the Banah algebra of all $A$- valued Borel- measurable functions $f: G\to A$ such that $\|f\|_{1, A}= \int_{G}\|f(x)\|_{A}dx<\infty$. Then
 $L^{1}(G, A)$ is
isometrically isomorphic with the projective tensor product $L^{1}(G)$ and $A$; See [\cite{kan}, Proposition 1.5.4]. Thus  $L^{1}(G, A)$  is a Banach algebra
under convolution product.   In \cite{AFR}, they showed that $L^{1}(G, A)$ is a group algebra if and only if $A$ is so. 
Due to  [\cite{kan}, Theorem 2.11.2], the character space of $L^{1}(G, A)$  is homeomorphic with $\hat{G}\times \Delta(A)$, such that for all $\chi\in\hat G$ and $\varphi\in\Delta(A)$,
$$
\chi\otimes\varphi: L^{1}(G)\hat\otimes A\to\mathbb C
$$
Where 
$$
\chi\otimes\varphi(f\otimes a) = \chi(f)\varphi(a)  \quad (f\in L^{1}(G), a\in A) 
$$
and if $f\in L^{1}(G, A)$, then
$$
\chi\otimes\varphi(f)= \varphi(\int_{G} \overline{\chi(x)}f(x)dx)= \int_{G} \overline{\chi(x)}\varphi(f(x))dx.
$$
In this paper, $\textup{BSE}$-  types and $\textup{BED}$- types  for  vector-valued group algebra $L^{1}(G, A)$ will be investigated. It will be shown that:

\begin{thm}
Let $A$ be a commutative semisimple unital separable dual Banach algebra and $G$ be a locally compact group. Then \\
(i) $L^{1}(G, A)$ is a ${\textup{BSE}}$-  algebra if and only if $A$ is a  ${\textup{BSE}}$ algebra.\\
(ii) $L^{1}(G, A)$  is a weak- ${\textup{BSE}}$  algebra if and only if $A$ is a  weak- ${\textup{BSE}}$ algebra and $G$ is discrete.\\
(iii)$L^{1}(G, A)$  is a ${\textup{BED}}$-  algebra if and only if $A$ is a  ${\textup{BED}}$-  algebra.\\
(iv) $L^{1}(G, A)$  is a weak- ${\textup{BED}}$  algebra if and only if $A$ is a  weak- ${\textup{BED}}$ algebra and $G$ is discrete.
\end{thm}
Furthermore, we show that:
$$
M(G, A)= L^{1}(G, A)
$$
if and only if $G$  is discrete.

%$$$$$$$$$$$$$$$$$$$$$$$$$$$$$$$$$$$$$$$$$$$$$$$$$$$$$$$$$$$$$$$$$$$$$$$$$$$$$$$$$$$$$$$$$$$$$$$$$$$$$$$$$$$$$$$$$$$$$$$$$$$$$$$$$$$$$$$$$$$$$$$$$$$$$$$$$$$$$$$
%+++++++++++++++++++++++++++++++++++++++++++++++++++++++++++++++++++++++++++++++++++++++++++++++++++++++++++++++++++++++++++++++++++++++++++++++++++++++
\section{Vector- valued measure algebras on groups}
Let $G$ be a locally compact group and $B$ be a Banach space. In this section, we will briefly introduce vector-valued measure algebra and the convolution in it.
 A function $\mu: \mathcal{B}(G)\to B$ is named a vector measure if it is $\sigma$- additive measure, where $\mathcal{B}(G)$ is the  $\sigma$- algebra of all Borel sets in $G$.

   The vector measure $\mu$ is named regular if for each $E\in \mathcal{B}(G)$ and $\epsilon >0$, there exists a compact set $K\subseteq G$  and an open set $U\subseteq G$ such that $K\subseteq E\subseteq U$ and $\|\mu(F)\|_{B}< \epsilon$,
for all Borel set $F\subseteq U\backslash K$. We define the total vacation of $\mu$, by $\|\mu\|= |\mu|(X)$, where
$$
|\mu|(E):= \textup{sup}\{\|\sum_{i=1}^{n}\mu(E_{i})\|: ~ \{E_{i}\} is ~ a ~ partion ~ of ~ E\}
$$
for all $E\in \mathcal{B}(G)$, which is a vector measure. The set of all bounded regular vector measures $\mu: \mathcal{B}(G)\to B$ will be denoted by $M(G, B)$. Let $\mu\in M(G, B)$ and $g\in B^{*}$. Then $go\mu:  \mathcal{B}(G)\to  \mathbb C$ is a regular measure and 
$$
B^{*}o M(G)= \{go\mu: g\in B^{*} ~ and ~ \mu\in M(G, B)\} = M(G)
$$ 
%^^^^^^^^^^^^^^^^^^^^^^^^^^^^^^^^^^^^^^^^^^^^^^^^^^^^^^^^^^^^^^^^^^^^^^^^^^^^^^^^^^^^^^^
\begin{thm}
Let $B$ be a  Banach space and $X$ be a locally compact Hausdorff space. Then
$$
B^{*}o M(X, B)= M(X).
$$
\end{thm}
\begin{proof}
If $g\in B^{*}$ and $\mu: B(X)\to B$ is a vector- measure, then for all sequence $(B_{n})$ of all pairwise disjoint Borel sets we have 
$$
\mu(\cup_{n=1}^{\infty} B_{n})= \sum_{n=1}^{\infty} \mu(B_{n})
$$
Let $S_{n}=  \sum_{k=1}^{n} \mu(B_{k})$ is in $A$. Since the series is uniformly convergent. So 
$$S_{n}\to S:= \mu(\cup_{n=1}^{\infty} B_{n}).$$
 Since $g$ is continuous, the following is the yield
$$
g(S_{n})= \sum_{k=1}^{n} go\mu(B_{k})\to g(S)
$$
Thus $go\mu$ is $\sigma$- additive. Moreover $\mu$ is a regular measure, thus for all $\epsilon> 0$, Borel set $E$ in $X$ there exist some compact set $K$ and open set $U$ such that $K\subseteq E\subseteq U$ where
$$
\|\mu(F)\|<\frac{\epsilon}{1+\|g\|}
$$
for all Borel set $F\subseteq U\backslash K$. Then
$$
|go\mu (f)|\leq \|g\|.\|\mu(F)\|< \epsilon
$$
so $go\mu$ is a regular measure, morevore $(B_{n})$ is a finite pairwise disjoint sequence of $\overline X$, then 
\begin{align*}
\|go \mu\| &:= |go\mu|(X)\\
                & = sup\{|\sum_{n=1}^{m}go\mu(B_{n})| : ~ (B_{n}) ~ is ~ pairwise~ disjoint\}\\
                 &\leq \|g\| sup\{\|\sum_{n=1}^{m}\mu(B_{n})\|  : ~ (B_{n}) ~ is ~ pairwise~ disjoint\}\\
                  &:= \|g\|.\|\mu\|
\end{align*}
Thus $go\mu$ is a bounded vector- measure and $\|go\mu\|\leq \|g\|.\|\mu\|$.

  Conversely, if $\nu\in M(X)$, then there exist $\mu\in M(X, B)$ and $g\in B^{*}$ such that $go\mu= \nu$. Because if $a\in B$ and $g\in B^{*}$ where $g(a)= 1$, define $\mu(E):= \nu(E)a$, then $\mu$
is a bounded regular measure in $M( X, B)$  such that
\begin{align*}
go\mu(E)= g(\nu(E)a)= \nu(E)g(a) = \nu(E)
\end{align*}
for all  Borel set $E$ in $X$. Therefore 
$$\nu \in B^{*}oM(X, B).$$
\end{proof}
%^^^^^^^^^^^^^^^^^^^^^^^^^^^^^^^^^^^^^^^^^^^^^^^^^^^^^^^^^^^^^^6
\begin{cor}
Let $A$ be a commutative semisimple Banach algebra and $G$ be a locally compact group. Then
$$
\Delta(A)o M(G, A)= M(G).
$$
\end{cor}
\begin{proof}
$$
\Delta(A)o M(G,A)\subseteq A^{*}oM(G, A)\subseteq M(G).
$$
Conversely, if $\nu\in M(G)$, $a\in A$ and $\varphi\in\Delta(A)$ where $\varphi(a)= 1$, then $\nu= \varphi o\mu$, such that $\mu(E)= \nu(E)a$, for $E\in B(G)$. Thus 
$$M(G)\subseteq \Delta(A)oM(G,A).$$
Therefore 
$$
\Delta(A)o M(G, A)= M(G).
$$
\end{proof}

Let $\Phi: C_{0}(G)\to B$ be a relatively compact linear operator. Then there exists a unique bounded regular vector measure $\mu\in M(G, B)$such that 
$$
\Phi(f)= \int_{G}fd\mu(x)
$$
for all $f\in C_{0}(G)$. This means that 
$$
go\Phi(f)= \int_{G}f(x)dgo\mu(x)
$$
for all $g\in B^{*}$
where $\|\Phi\|= |\mu|(G)$.

  Furthermore, for each vector measure $\mu: \mathcal{B}(G)\to B$  the linear operator defined  by
$$
\Phi_{\mu}(f)= \int_{B}fd\mu
$$
for  $f\in C_{0}(G)$, is a relatively weakly compact bounded operator in $W(C_{0}(G), B)$. Therefore
$$
M(G, B)= W(C_{0}(G), B)
$$
 as two Banach spaces such that the maps $\mu\mapsto \Phi_{\mu}$ is isometric. Hence $(M(G, B), \|.\|)$ is a Banach space; see [\cite{kos}, Lemma2].

  Let $C_{0}(G)\hat{\otimes} B$, [resp $C_{0}(G)\check{\otimes} B$] be projective [resp, injective] Banach spaces. Then
$$
(C_{0}(G)\hat{\otimes} B)^{*}= M(G, B^{*})
$$
as two Banach spaces such that the map $\mu\mapsto \Phi_{\mu}$ defined by 
$$
\Phi_{\mu}(f\otimes z)= \int_{G}\int_{{\mathcal{B}}_{1}(B^{*})}f(x)g(z)d\mu^{'}(x, g)
$$
is isometric, for each $f\in C_{0}(G)$ and $z\in B$, where $\mu^{'}: \mathcal{B}(G\times {\mathcal{B}}_{1}(B^{*}))\to \mathbb C$
is the corresponding regular vector measure on $G\times {\mathcal{B}}_{1}(B^{*})$, where 
$$
{\mathcal{B}}_{1}(B^{*}):= \{ g\in B^{*}: \|g\|\leq 1\}.
$$
to the linear functional $\Phi_{\mu}$.
It is to be noted that  for  Banach spaces $X$ and $Y$  the following is the yield:
$$
(X\check \otimes Y)^{*}= M({\mathcal{B}}_{1}(X^{*})\times {\mathcal{B}}_{1}(Y^{*}))
$$
as two Banach spaces, where the map $\mu\mapsto F_{\mu}$ is isometric, where 
$$
F_{\mu}(x\otimes y)= \int_{{\mathcal{B}}_{1}(X^{*})}\int_{{\mathcal{B}}_{1}(Y^{*})}f(x)g(y)d\mu(f,g);
$$
See[\cite{JD}, Theorem1.11].
In particular, let $A= E^{*}$ be a dual Banach algebra. Then by using Bochner's theorem on semigroup one can show that:
\begin{align*}
(C_{0}(G\check\otimes E))^{*} &= {C_{0}(G, E)}^{*}\\
                                                  &= M({\mathcal{B}}_{1}(M(G))\times {\mathcal{B}}_{1}(A))\\
                                                  &= M(G\times A_{1})
\end{align*}
where $A_{1}$ is unit ball of $A$
as Banach algebra such that the map $\mu\mapsto \Phi_{\mu}$ defined by 
$$
\Phi_{\mu}(f\otimes e)=\int_{G}\int_{A_{1}}f(x)a(e)d\mu(x, a)
$$
from $ M(G\times A_{1})$ onto ${C_{0}(G, E)}^{*}$
is isometric, for each $f\in C_{0}(G)$ and $e\in E$; See [\cite{PR}, p.316] . Let $\mu,\nu\in M(G\times A_{1})$ and $B\in \mathcal{B}(G)$. Then:
$$
\mu\star\nu(B)(e)= \int_{G\times A_{1}}\int_{G\times A_{1}}\chi_{B}(xy)ab(e)d\mu(x, a)d\nu(y, b)
$$
 is the vector measures corresponding to the linear functional 
$$
\Phi_{\mu}\star \Phi_{\nu}(f\otimes e)= \int_{G\times A_{1}}\int_{G\times A_{1}} f(xy)ab(e)d\mu(x, a)d\nu(y, b)
$$
 in $(C_{0}(G, E))^{*}$. Therefore $(M(G, A), \star)$ with predual $C_{0}(G, E)$ is a Banach algebra.

  The Banach space $X$ is a Pillips space, if for every compact Hausdorff space $Y$, every Radon measure $\mu: \mathcal{B}(Y)\to \mathbb C$ and bounded linear operator  $T: L^{1}(Y,\mu)\to X$, there exists 
a strongly measurable function $f: Y\to X$ where
$$
T(g) = \int_{Y}g(y)f(y)d\mu(y)
$$
for all $g\in L^{1}(Y, \mu)$ and 
$$
es sup\{\|f(y)\|: y\in Y\}= \|T\|
$$
see[\cite{JL}, p.104]. It is to be noted that each reflexive Banach space and each dual Banach space which is separable is Philips space.
   Let $X$ be a Philips space. Then 
$$
M(G, X)= M(G)\hat \otimes X
$$
isometrically.

  In general, 
$$
(X\hat\otimes Y)^{*}= X^{*}\check\otimes Y^{*}
$$
where $X^{*}$ or $Y^{*}$ is  a Philips space; see[\cite{JL}, 5.3]. Furthermore, $M(G)\hat\otimes X$ can be embedded in $M(G, X)$ isometrically. The equality holds if $X$ has Radon Nikodim property.
In particular, $M(G, A) = M(G)\hat \otimes A$ isometrically isomorphism as Banach algebras, where $A$ is a separable dual Banach algebra; see  \cite{kan}.
It is to be noted that the Banach algebra  $A$ has Radon Nikodim property if and only if for each bounded set  $E\subseteq A$ and $\epsilon> 0$, there exists $b\in E$ such that $b\notin\overline{<E\backslash B(b, r)>}$; See [\cite{D}, p. 819].
%^^^^^^^^^^^^^^^^^^^^^^^^^^^^^^^^^^^^^^^^^^^^^^^^^^^^^^^^^^^^^^^
\begin{thm}
Let $A$ be a  Banach algebra and $G$ be a locally compact group. Then 
$$
A^{*}o L^{1}(G, A)= L^{1}(G).
$$
\end{thm}
\begin{proof}
If $f\in L^{1}(G)$ and $a\in A$ is nonzero, then there exists $g\in A^{*}$ where $g(a)= 1$. Assume that $h= f\otimes a$, so $h\in L^{1}(G, A)$ and $\|h\|_{1, A}= \|f\|_{1}. \|a\|_{A}$
such that 
$$
goh(x)= g(f(x)a)= f(x)g(a)= f(x) \quad (x\in G)
$$ 
Hence $f=goh$.
At a result $L^{1}(G)\subseteq A^{*}oL^{1}(G, A)$.

  Conversely,  if $g\in A^{*}$ and $h\in L^{1}(G, A)$, then 
$$
|goh(x)|\leq\|g\|.\|h(x)\|_{A}
$$
and  so
$$
\|goh\|_{1}\leq \|g\|.\|h\|_{1, A}.
$$
Then 
$$A^{*}o L^{1}(G, A)\subseteq L^{1}(G).$$
  Therefore 
$$
A^{*}o L^{1}(G, A)= L^{1}(G).
$$
\end{proof}
%^^^^^^^^^^^^^^^^^^^^^^^^^^^^^^^^^^^^^^^^^^^^^^^^^^^^^^^^^^^^^^
\begin{cor}
Let $A$ be a semisimple commutative Banach algebra and $G$ be a locally compact group. Then 
$$
\Delta(A)o L^{1}(G, A)= L^{1}(G).
$$
\end{cor}
%^^^^^^^^^^^^^^^^^^^^^^^^^^^^^^^^^^^^^^^^^^^^^^^^^^^^^^^^^^^^^^^^^^^^^^^^^^^^^^^^^
\begin{lem}\label{gdis}
If $A$ is  unital dual Banach  algrbra, then 
$$
M(G, A)= L^{1}(G, A)
$$
if and only if  $G$ is discrete.
\end{lem}
\begin{proof}
Assume that $M(G, A)= L^{1}(G, A)$.
If $\mu\in M(G)$, then $\mu\otimes 1_{A}\in L^{1}(G, A)$. Then
$$
\mu\otimes 1_{A}= \sum_{n=1}^{\infty}\mu_{n}\otimes a_{n}
$$
where $\mu_{n}\in L^{1}(G)$, $a_{n}\in A= E^{*}$ and $\sum_{n=1}^{\infty}\|\mu_{n}\|.\|a_{n}\|<\infty$. In another hand 
$$M(G, A)= C_{0}(G, E)^{*}$$
 so
\begin{align*}
\mu\otimes 1_{A}(f\otimes e)= \mu(f)1_{A}(e)= \sum_{n=1}^{\infty} \mu_{n}(f)a_{n}(e)
\end{align*}
for all $f\in C_{0}(G)$ and fixed $e\in E$. At a result 
\begin{align*}
\mu(f) &= \sum_{n=1}^{\infty}\frac{a_{n}(e)}{1_{A}(e)}\mu_{n}(f)  
\end{align*}
Put 
$$\mu:=  \sum_{n=1}^{\infty}c_{n}\mu_{n}\in L^{1}(G)$$
 where $c_{n}= \frac{a_{n}(e)}{1_{A}(e)}$. This implies that $L^{1}(G)= M(G)$ and then $G$ is discrete.

  If $G$ is discrete, then
\begin{align*}
 L^{1}(G, A) &= l_{1}(G, A)\\
                     &= l_{1}(G)\hat\otimes A\\
                       &= M(G) \hat\otimes A = M(G, A)
\end{align*}
\end{proof}

%%%%%%%%%%%%%%%%%%%%%%%%%%%%%%%%%%%%%%%%%%%%%%%%%%%%%%%%%%%%%%%%%%%%%%%%%%%%%%%%%%%%%%%%%%%%%%%%%%%%%%%%%%
\section{\textup{BSE}- properties  of $L^{1}(G, A)$}
Let  $X$ be a non-empty set,  $A$ be a commutative Banach algebra, and $1\leq p< \infty$. The space  of all functions such $f: X\to A$ which 
$$
\sum_{x\in X}\|f(x)\|^{p}<\infty
$$
 is expressed by $l^{p}(X, A)$, with pointwise product and the norm
$$
\|f\|_{p}= (\sum_{x\in X}\|f(x)\|^{p})^{\frac{1}{p}}
$$
is a commutative Banach algebra.The
researcher in \cite{AP}, revealed that  $l^{p}(X, A)$ is a $\textup{BED}$- algebra if and only if $A$ is so. They proved that  $l^{p}(X, A)$ is a weak- $\textup{BSE}$ algebra if and only if $A$ is so.
In \cite{MAR}, the authers studied $\textup{BED}$ and $\textup{BSE}$- properties of vector- valued $C^{*}$- algebras $C_{0}(X, A)$.
The researchers in \cite{AR} proved that if $A$ is an unital BSE Banach algebra and $G$ is an Abelian locally compact  group, then 
$$
C_{BSE}(\Delta(L^{1}(G, A)))= {M(G, C_{BSE}(\Delta(A)\widehat)}
$$
In particular, $L^{1}(G, A)$ is a $\textup{BSE}$- algebra if and only $A$ is so, where $A$ is a semisimple commutative unital Banach algebra.
If $A$ is a separable dual Banach algebra with predual $E$, then 
$$M(G, A) = C_{0}(G, E)^{*} = M(G)\hat\otimes A
$$
 as a Banach algebra and  isometric such that $C_{0}(G, E)$ is predual of $M(G,A)$.
%^^^^^^^^^^^^^^^^^^^^^^^^^^^^^^^^^^^^^^^^^^^^^^^^
\begin{lem}
If $A = E^{*}$ is a dual Banach algebra, then
$$
 M(G, A) = C_{0}(G, E)^{*}
$$
is a
dual Banach algebra.
\end{lem}
\begin{proof}
Since, $M(G) = C_{0}(G)^{*}$ is a dual algebra, then $C_{0}(G) * M(G)\subseteq  M(G)$ and
$M(G) * C_{0}(G) \subseteq C_{0}(G)$. So
$$
C_{0}(G, E) * M(G, A) \subseteq C_{0}(G, E).
$$
and 
$$
M(G, A) * C_{0}(G, E)\subseteq C_{0}(G, E).
$$
This complete the proof.
\end{proof}
The next Lemma is applied in some further results. 
%^^^^^^^^^^^^^^^^^^^^^^^^

\begin{lem}\label{Lsv}
Let $A$ be a commutative Banach algebra and $G$ be an Abelian locally compact group. Then\\
(i) $L^{1}(G, A)$ is a Banach algebra under convolution product .\\
(ii) $L^{1}(G, A)$ has a bounded $\Delta$- weak approximate identity if and only if $A$ has so.\\
(iii) $L^{1}(G, A)$ has a bounded approximate identity if and only if $A$ has a bounded approximate identity.\\
(iv) $L^{1}(G, A)$ is semisimple if and only if $A$ is semisimple.

\end{lem} 
\begin{proof}
(i) According to [\cite{kan}, Proposition 1.5.4.],  $L^{1}(G, A)$ is
isometrically isomorphic with $L^{1}(G)\widehat\otimes A$, then $L^{1}(G, A)$ is a Banach algebra under convolution product.\\
(ii) By Proposition 3.2. of \cite{AR}, $L^{1}(G, A)$ has a bounded $\Delta$- weak approximate identity if and only if $A$ has so.\\
(iii) If $(f_{\gamma})$ is a bounded approximate identity for $L^{1}(G)$ and $(a_{\alpha})$ is a bounded approximate identity for $A$, then $(f_{\gamma}\otimes a_{\alpha})$ is a bounded approximate identity for $L^{1}(G, A)$.
Conversely, if $(h_{\beta})$ is a bounded approximate identity for  $L^{1}(G, A)$, then $(h_{\beta}(x))$ is a bounded approximate identity for $A$.\\
(iv) Due to Corollary 2.11.3 and Theorem 2.11.8 of \cite{kan}, $L^{1}(G, A)$ is semisimple if and only if $A$ is semisimple; See \cite{D}.

\end{proof}

%^^^^^^^^^^^^^^^^^^^^^^^^^^^^^^^^^^^^^^
\begin{thm}
Let $A$ be a semisimple commutative Banach algebra and $G$ be a locally compact group. Then 
$$
{ L^{1}(G, A)}^{*}= L^{\infty}(G, A^{*}).
$$
Isometrically as Banach spaces.
\end{thm}
\begin{proof}
If 
\begin{align*}
F:  L^{\infty}(G, A^{*}) &\to  { L^{1}(G, A)}^{*}\\
                                         & g\mapsto F_{g}
\end{align*}
where $F_{g}(f)= \int_{G} g(x)(f(x))dx$, for $f\in L^{1}(G, A)$. Thus 
\begin{align*}
\mid F_{g}(f)\mid &\leq \int_{G}\mid g(x)(f(x))\mid dx\\
                 &\leq \int_{G}\|g(x)\|_{A^{*}}.\|f(x)\|_{A}dx\\
                 &\leq \|g\|_{\infty, A^{*}}.\|f\|_{1,A}
\end{align*}
 It is clear that $F_{g}$ is linear. So
$$
F_{g}\in  { L^{1}(G, A)}^{*}, ~ and~ , ~ \|F_{g}\|\leq \|g\|_{\infty, A^{*}}
$$
If $P\in { L^{1}(G, A)}^{*}$ and $a\in A$, then for $h\in L^{1}(G)$ we define:
$$
P_{a}(h):= P(h\otimes a)
$$
Thus $P_{a}\in {L^{1}(G)}^{*}= L^{\infty}(G)$. Therefore there exists $g_{a}\in L^{\infty}(G)$ such that
$$
P(h\otimes a)= \int_{G}g_{a}(x)h(x)dx.
$$
Define $g(x)(a) := g_{a}(x)$, for $x\in G$. Then $g\in  L^{\infty}(G, A^{*})$ and $F_{g}= P$, thus $F$ is surjective. If $f\in  L^{1}(G, A)$, then there exist some sequence $(f_{n})$ in $L^{1}(G)$ and $(a_{n})$ in $A$ such that
$$
f= \sum_{n=1}^{\infty}f_{n}\otimes a_{n}
$$
where  $ \sum_{n=1}^{\infty}\|f_{n}\|_{1}.\|a_{n}\|_{A}<\infty$. The following is the yield:
\begin{align*}
P(f) &= \underset{n}{lim}P(\sum_{k=1}^{n}f_{k}\otimes a_{k})\\
      &= \underset{n}{lim}\sum_{k=1}^{n}F_{g}(f_{k}\otimes a_{k})\\
     &= \underset{n}{lim}\sum_{k=1}^{n}\int_{G}g(x)(a_{k})f_{k}(x)dx\\
      &= \underset{n}{lim}\int_{G}g(x)(\sum_{k=1}^{n}f_{k}\otimes a_{k}(x))dx\\
        &= \int_{G}g(x)(f(x))dx\\
       &= F_{g}(f)
\end{align*}
Therefore $F_{g}=P$. In another hand $\|P_{a}\|=\|g_{a}\|_{\infty}$, for all $a\in A_{1}$. Therefore
\begin{align*}
|F_{g}(h\otimes a)| &= |P(h\otimes a)|=|P_{a}(h)| 
\end{align*}
for $h\in L^{1}(G)$ and so 
$$
\|F_{g}\|\geq sup\{ |P_{a}(h)|: \|h\|_{1}\leq 1\}= \|g_{a}\|_{\infty}
$$
Thus 
$$
\|F_{g}\|\geq sup\{|g(x)(a)|: x\in G, \|a\|\leq 1\}= \|g\|_{\infty, A^{*}}
$$
As a result $\|F_{g}\|= \|g\|_{\infty, A^{*}}$, for all $g\in  L^{\infty}(G, A^{*})$ and so $F$ is isometry. Therefore $F$ is an isometric isomorphism map.
\end{proof}

In this section, the correlations between the \textup{BSE}- property and \textup{BED}- property of Banach algebra $A$ and $L^{1}(G, A)$ are assessed.
%^^^^^^^^^^^^^^^^^^^^^^^^^^^^^^^^^^^^^^^^^^^^^^^^^^^^^^^^^^^^^^^^^^^^^^^^^^^^^^^
%^^^^^^^^^^^^^^^^^^^^^^^^^^^^^^^^^^^^^^^^^^^^^^^^^^^^^^^^^^^^^^^^^^^^^^^^^^^^^^^^^^^^^^^^^^^^^^^^^^^^^^^^^^^^^^^^^^^^^^
%\begin{lem}\label{ap}
%Let $A$ be a Banach algebra. Then $L^{1}(G, A)$ has a bounded approximate identity(B.A.I) iff $A$ has B.A.I.
%\end{lem}
%\begin{proof}
%Assume that $(e_{\alpha})$ is B.A.I for $A$. If $(f_{\beta})$ is B.A.I for $L^{1}(G)$, then $(f_{\beta}\otimes e_{\alpha})$ is B.A.I for  $L^{1}(G,A)= L^{1}(G)\hat\otimes A$.
%
%  If $(g_{\gamma})$ is B.A.I for $L^{1}(G,A)$ and $x\in G$, then $(g_{\gamma}(x))$ is B.A.I for $A$. Since if $f\in L^{1}(G)$, then 
%\begin{align*}
%\|f*g_{\gamma}(x)- f\|_{1} &= \int_{G}|f*g_{\gamma}(x)(y)-f(y)|dy\\
%                                             &=\int_{G}|\int_{G} f(y)g_{\gamma}(y^{-1}x)- f(y)|dy
%\end{align*}
%but $g_{\gamma}*h \overset{\|.\|_{A}}{\to}h$, for all $h\in L^{1}(G,A)$. Thus
%$$
%\| g_{\gamma}*h- h\|_{1,A}= \int_{G} \|g_{\gamma}*h(z)- h(z)\|_{A}dz\to 0
%$$
%In a special case, $h(z) = f\otimes a$ the following is yield:
%$$
%\|(f\otimes a)*g_{\gamma}- f\otimes a\|_{1,A}
%$$
%\end{proof}
%^^^^^^^^^^^^^^^^^^^^^^^^^^^^^^^^^^^^^^^^^^^^^^^^^^^^^^^^^^^^^^^^^^^^^^^^^^^^^
\begin{lem}\label{rel1}
(i) If $g\in L^{1}(G)$ and $\|g\|_{1}\leq 1$, then
$$
\|\sum_{i=1}^{n} c_{i}\chi_{i}(g)\varphi_{i}\|_{A^{*}}\leq \|\sum_{i=1}^{n} c_{i} (\chi_{i}\otimes\varphi_{i})\|_{{L^{1}(G, A)}^{*}}
$$
(ii)
If  $\chi_{0}\in\hat G$ and $\varphi_{i}\in\Delta(A)$, then 
$$
\|\sum_{i=1}^{n} c_{i}(\chi_{0}\otimes \varphi_{i})\|_{{L^{1}(G,A)}^{*}}\leq  \|\sum_{i=1}^{n} c_{i} \varphi_{i}\|_{A^{*}}
$$ 
(iii) If $\sigma \in C_{\textup{BSE}}(\Delta(A))$, then 
$$
\|\sum_{i=1}^{n} c_{i}\sigma(\varphi_{i})\chi_{i}\|_{{L_{1}(G)}^{*}}\leq \|\sigma\|_{\textup{BSE}}\|\sum_{i=1}^{n} c_{i}(\chi_{i}\otimes\varphi_{i})\|_{{L^{1}(G, A)}^{*}}
$$ 
\end{lem}
\begin{proof}
(i)
\begin{align*}
\|\sum_{i=1}^{n} c_{i} (\chi_{i}\otimes\varphi_{i})\|_{{L^{1}(G, A)}^{*}} &= \textup{sup}\{ |\sum_{i=1}^{n} c_{i}(\chi_{i}\otimes \varphi_{i})(h)| :~ h\in L^{1}(G,A), \|h\|_{1,A}\leq 1\}\\
                                                                                                                       &\geq \textup{sup}\{ |\sum_{i=1}^{n} c_{i}(\chi_{i}\otimes \varphi_{i})(g\otimes a)| :~  \|a\|_{A}\leq 1\}\\
                                                                                                                        &= \textup{sup}\{ |\sum_{i=1}^{n} c_{i}\chi_{i}(g) \varphi_{i}( a)| :~  \|a\|_{A}\leq 1\}\\
                                                                                                                          &= \|\sum_{i=1}^{n} c_{i}\chi_{i}(g) \varphi_{i}\|_{A^{*}}
\end{align*}
(ii)
Since
$$\|\chi_{0}(h)\|\leq \int_{G}|\overline{\chi_{0}(x)}|.\|h(x)\|_{A}\leq \|h\|_{1,A}.$$
Thus
\begin{align*}
\|\sum_{i=1}^{n} c_{i}(\chi_{0}\otimes \varphi_{i})\|_{{L^{1}(G,A)}^{*}} &= \textup{sup}\{ |\sum_{i=1}^{n} c_{i}(\chi_{0}\otimes \varphi_{i})(h)| :~ h\in L^{1}(G,A), \|h\|_{1,A}\leq 1\}\\
                                                                                                                            &=  \textup{sup}\{ |(\sum_{i=1}^{n} c_{i} \varphi_{i})(\chi_{0}(h))| :~ h\in L^{1}(G,A), \|h\|_{1,A}\leq 1\}\\
                                                                                                                            &\leq     \textup{sup}\{ |\widehat{\chi_{0}(h)}(\sum_{i=1}^{n} c_{i} \varphi_{i})| :~ h\in L^{1}(G,A), \|h\|_{1,A}\leq 1\}\\
                                                                                                                               &\leq     \textup{sup}\{ \|{\chi_{0}(h)}\|.\|\sum_{i=1}^{n} c_{i} \varphi_{i}\|_{A^{*}} :~ h\in L^{1}(G,A), \|h\|_{1,A}\leq 1\}\\
                                                                                                                                &\leq \|\sum_{i=1}^{n} c_{i} \varphi_{i}\|_{A^{*}}
\end{align*}

(iii)
\begin{align*}
\|\sum_{i=1}^{n} c_{i}\sigma(\varphi_{i})\chi_{i}\|_{{L_{1}(G)}^{*}} &= sup \{ \mid \sum_{i=1}^{n} c_{i}\sigma(\varphi_{i})h(\chi_{i})\mid ~: \|h\|_{1}\leq 1\}\\
                                                                                                                   & \leq \|\sigma\|_{\textup{BSE}}. sup \{ \|\sum_{i=1}^{n} c_{i}h(\chi_{i})\varphi_{i}\|_{\infty} ~: \|h\|_{1}\leq 1\}
\end{align*}
But 
\begin{align*}
\|\sum_{i=1}^{n} c_{i}h(\chi_{i})\varphi_{i}\|_{\infty} &= sup\{ \mid\sum_{i=1}^{n} c_{i}h(\chi_{i})\varphi_{i}(a)\mid : \|a\|_{A}\leq 1\}\\
                                                                                          &= sup\{ \mid\sum_{i=1}^{n} c_{i}(\chi_{i}\otimes \varphi_{i})(h\otimes a)\mid : \|a\|_{A}\leq 1\}
\end{align*}
Thus 
\begin{align*}
\|\sum_{i=1}^{n} c_{i}\sigma(\varphi_{i})\chi_{i}\|_{{L_{1}(G)}^{*}} &\leq \|\sigma\|_{\textup{BSE}}. \underset{\|a\|\leq 1}{sup}\underset{\|h\|_{1}\leq 1}{sup} \{\|h\|_{1}.\|a\|. \|\sum_{i=1}^{n} c_{i}(\chi_{i}\otimes \varphi_{i})\|_{{L^{1}(G, A)}^{*}}\}\\
                                                                                                                  &\leq  \|\sigma\|_{\textup{BSE}}.\|\sum_{i=1}^{n} c_{i}(\chi_{i}\otimes \varphi_{i})\|_{{L^{1}(G, A)}^{*}}
\end{align*}
This completes the proof.
\end{proof}
%^^^^^^^^^^^^^^^^^^^^^^^^^^^^^^^^^^^^^^^^^^^^^
\begin{lem}\label{l2.2}
If $A$ is a seperable dual Banach algebra and $\mu\in M(G, A)$ such that 
$$\hat{\mu}\in C_{\textup{BSE}}^{0}(\Delta(L^{1}(G, A)))$$
 then $\mu\in  L^{1}(G,A)$.
\end{lem}
\begin{proof}
Let $\hat{\mu}\in C_{\textup{BSE}}^{0}(\Delta(L^{1}(G, A))$. Then for all $\epsilon>0$, there exists some compact set $K\subset\hat{G}\times \Delta(A)$ where for all $c_{i}\in\mathbb C$ and $\psi_{i}= \chi_{i}\otimes \varphi_{i}\notin K$ the following is yield:
$$
|\sum_{i=1}^{n} c_{i}\hat{\mu}(\chi_{i}\otimes\varphi_{i})|<\frac{\epsilon}{2}
$$
where 
$$\|\sum_{i=1}^{n} c_{i}(\chi_{i}\otimes\varphi_{i})\|_{ L^{1}(G,A)^{*}}\leq 1$$
 Assume that $\mu= \sum_{n=1}^{\infty}\mu_{n}\otimes a_{n}$ where $\mu_{n}\in M(G)$, $a_{n}\in A$  and $\sum_{n=1}^{\infty}\|\mu_{n}\|\|a_{n}\|<\infty$. Now, we will 
proved that $\mu_{n}\in L^{1}(G)$ and $\mu\in  L^{1}(G,A)$. We have:
\begin{align*}
|\sum_{i=1}^{n}c_{i}\hat{\mu}(\chi_{i}\otimes\varphi_{i})| &= |\sum_{i=1}^{n}c_{i}\lim_{m}\sum_{j=1}^{m}\widehat{\mu_{j}\otimes a_{j}}(\chi_{i}\otimes\varphi_{i})|\\
                                                                                                   &= \lim_{m}|\sum_{i=1}^{n}\sum_{j=1}^{m}c_{i}\hat{\mu_{j}}(\chi_{i})\varphi_{i}(a_{j})|< \frac{\epsilon}{2}
\end{align*}
As a  result, there exists some integer number $N$ such that for all $m\geq N$ where 
$$|\sum_{i=1}^{n}\sum_{j=1}^{m}c_{i}\hat{\mu_{j}}(\chi_{i})\varphi_{i}(a_{j})|< \epsilon$$
such that
 $$\|\sum_{i=1}^{n}\chi_{i}\otimes \varphi_{i}\|_{{L^{1}(G, A)}^{*}}\leq 1.$$
If $k\in\mathbb N$, $\varphi_{k}(a_{k})=1$, $\varphi_{k}(a_{j})= 0$ for $1\leq j\leq m, j\neq k$ where $m\geq k+N$. Set $\varphi_{i}= \varphi_{k}$ for $1\leq i\leq n$. Then
$$
|\sum_{i=1}^{n}\sum_{j=1}^{m}c_{i}\hat{\mu_{j}}(\chi_{i})\varphi_{i}(a_{j})|<\epsilon
$$
This means that
\begin{align*}
|\sum_{i=1}^{n}c_{i}\hat{u_{k}}(\chi_{i})| &\leq\epsilon \|\sum_{i=1}^{n}c_{i} (\chi_{i}\otimes \varphi_{k})\|_{ L^{1}(G,A)^{*}}\\
                                                                         &\leq \epsilon \|\sum_{i=1}^{n}c_{i} \chi_{i}\|_{ {L^{1}(G)}^{*}}
\end{align*}
Thus 
$$\hat{\mu_{k}}\in C_{\textup{BSE}}^{0}(\Delta(L^{1}(G)))= L^{1}(G\widehat)$$
 so 
$$\hat{\mu_{k}}\in {L^{1}(G\widehat)}$$
 and
$$\mu= \sum_{k=1}^{\infty}\mu_{k}\otimes a_{k}$$
 such that $\sum_{k=1}^{\infty}\|\mu_{k}\|\| a_{k}\|<\infty$.
 Therefore 
$$
\mu\in L^{1}(G)\hat\otimes A= L^{1}(G, A)
$$
\end{proof}
%^^^^^^^^^^^^^^^^^^^^^^^^^^^^^^^^^^^^^^^^^^^^^^^
\begin{lem}\label{lfc0}

Let $A$ be a commutative semisimple unital  Banach algebra and $G$ be a locally compact group.\\
(i)
 If $\mu\in M(G)$ and $\sigma\in  C_{\textup{BSE}}(\Delta(A))$, then ${\hat \mu}\otimes \sigma\in C_{\textup{BSE}}(\Delta(L^{1}(G,A))$.\\
(ii)
If $f\in L^{1}(G)$ and $\sigma\in  C_{\textup{BSE}}^{0}(\Delta(A))$, then ${\hat f}\otimes \sigma\in C_{\textup{BSE}}^{0}(\Delta(L^{1}(G,A))$.
\end{lem}
\begin{proof}
(i)
Assume that $c_{i}\in\mathbb C$, $\varphi_{i}\in\Delta(A)$ and $\chi_{i}\in \hat G$. Then by applying \ref{rel1}, the following is yield: 
\begin{align*}
|\sum_{i=1}^{n}c_{i}{\hat \mu}\otimes \sigma(\chi_{i}\otimes\varphi_{i})| &= |\sum_{i=1}^{n}c_{i}\hat{\mu}(\chi_{i})\sigma(\varphi_{i})|\\
                                                                                       &= \mid \hat{\mu}(\sum_{i=1}^{n}c_{i}\sigma(\varphi_{i})\chi_{i})|\\
                                                                                       &= \mid \int_{G}\sum_{i=1}^{n}c_{i}\sigma(\varphi_{i})\overline{\chi_{i}(x)}d\mu(x)\mid\\
                                                                                      &\leq\|\sigma\|_{\textup{BSE}}\|\mu\|\|\sum_{i=1}^{n}c_{i}(\chi_{i}\otimes\varphi_{i})\|_{ L^{1}(G,A)^{*}}
\end{align*}
Thus ${\hat \mu}\otimes\sigma\in C_{\textup{BSE}}(\Delta(L^{1}(G,A))$.\\
(ii)
 Let $f\in L^{1}(G)$. Then $\hat{f}\in L^{1}(G\widehat)= C_{\textup{BSE}}^{0}(\Delta(L^{1}(G)))$. So for all $\epsilon>0$ there exist some compact set $F\subseteq \hat G$ such that for all $c_{i}\in \mathbb C$ and $\chi_{i}\notin F$,
the following is the yield:
$$
|\sum_{i=1}^{n}c_{i}{\hat f}(\chi_{i})|< \epsilon\|\sum_{i=1}^{n}c_{i}\chi_{i}\|_{\infty}.
$$
Set $K:= F\times \Delta(A)$, so by applying \ref{rel1} part (iii), for each $\chi_{i}\otimes \varphi_{i}\notin K$, we have
\begin{align*}
|\sum_{i=1}^{n}c_{i}{\hat f}\otimes \sigma(\chi_{i}\otimes\varphi_{i})| &= |\sum_{i=1}^{n}c_{i}\hat{f}(\chi_{i})\sigma(\varphi_{i})|\\
                                                                                       &\leq\epsilon \|\sum_{i=1}^{n}c_{i}\sigma(\varphi_{i})\chi_{i}\|_{{L^{1}(G)}^{*}}\\
                                                                                      &\leq\epsilon \|\sigma\|_{\textup{BSE}}\|\sum_{i=1}^{n}c_{i}(\chi_{i}\otimes\varphi_{i})\|_{ L^{1}(G,A)^{*}}
\end{align*}
Thus ${\hat f}\otimes\sigma\in C_{\textup{BSE}}^{0}(\Delta(L^{1}(G,A))$.
\end{proof}
%^^^^^^^^^^^^^^^^^^^^^^^^^^^^^^^^^^^^^^^^^^^^^^^^^^^^^^^^^^^^^^^^^
\begin{lem}\label{lbed}
Let $A$ be a semisimple commutative Banach algebra and $G$ be a locally compact group.
If $L^{1}(G, A)$  is a $\textup{BED}$- algebra, then  $A$ is a  $\textup{BED}$- algebra.
\end{lem}
\begin{proof}
If $\chi_{0}\in \hat G$,  $f\in L^{1}(G)$ such that $\chi_{0}(f)=1$ and $\sigma\in  C_{\textup{BSE}}^{0}(\Delta(A))$, then  by applying  Lemma \ref{Lsv}, we have
$$\hat{f}\otimes\sigma\in  C_{\textup{BSE}}^{0}(\Delta(L^{1}(G, A)))= {L^{1}(G, A\widehat )}.$$
Thus there exists some $g\in L^{1}(G, A)$ where $\hat{f}\otimes\sigma= \hat g$, so
\begin{align*}
\sigma(\varphi)  &= f(\chi_{0})\sigma(\varphi)\\
                           &= \hat{g}(\chi_{0}\otimes\varphi)= \widehat{\chi_{0}(g)}(\varphi)
\end{align*}
and so $\sigma =  \widehat{\chi_{0}(g)}\in \hat A$. Therefore 
$$C_{\textup{BSE}}^{0}(\Delta(A))\subseteq \hat A.$$

    For the reverse inclusion, if $a\in A$, $\chi_{0}\in \hat G$ and  $f\in L^{1}(G)$ such that $\chi_{0}(f)=1$, then 
$$\hat{f}\otimes \hat{a}\in  C_{\textup{BSE}}^{0}(\Delta(L^{1}(G, A))).$$ 
So  for all $\epsilon> 0$, there exists some compact set $K\subseteq \hat{G}\otimes \Delta(A)$ where  the following inequality holds:
$$
 |\sum_{i=1}^{n} c_{i}\hat{f}\otimes \hat{a}(\chi_{i}\otimes \varphi_{i})|\leq \epsilon\|\sum_{i=1}^{n} c_{i}(\chi_{i}\otimes \varphi_{i})\|_{{L^{1}(G,A)}^{*}}
$$
when $\chi_{i}\otimes \varphi_{i}\in \hat{G}\otimes \Delta(A)\backslash K$. If 
$$
S: ~ \hat{G}\otimes \Delta(A)\to \Delta(A)
$$
  where 
$$
\chi\otimes\varphi\mapsto \varphi
$$
Thus $S$ is continuous. Set $F:= S(K)$ is compact. If $\varphi_{i}\in\Delta(A)\backslash F$, then $\chi_{0}\otimes \varphi_{i}\notin K$. This implies that
\begin{align*}
|\sum_{i=1}^{n} c_{i}\varphi_{i}(a)| &= |\sum_{i=1}^{n} c_{i}(\hat{f}\otimes \hat{a})(\chi_{0}\otimes\varphi_{i})|\\
                                                              & \leq \epsilon\|\sum_{i=1}^{n} c_{i}(\chi_{0}\otimes \varphi_{i})\|_{{L^{1}(G,A)}^{*}}\\
                                                               & \leq \epsilon\|\sum_{i=1}^{n} c_{i} \varphi_{i}\|_{A^{*}}.
\end{align*}
Then $\hat{a}\in C_{\textup{BSE}}^{0}(\Delta(A))$ and so $\hat A\subseteq C_{\textup{BSE}}^{0}(\Delta(A))$. Therefore 
$$
C_{\textup{BSE}}^{0}(\Delta(A))= \hat A
$$
This means that $A$ is a $\textup {BED}$- algebra.

\end{proof}
%^^^^^^^^^^^^^^^^^^^^^^^^^^^^^^^^^^^^^^^^^^^^^^^^^^^^^^^^^^^^^^^^^^^^^^^^^^^^^^
\begin{rem}
If $A$ is unital and $ L^{1}(G, A)$ is $\textup{BED}$- algebra, then $ L^{1}(G, A)$ has a bounded approximate identity. Thus it has $\Delta$- weak bounded approximate identity, so $ L^{1}(G, A)$ is a $\textup{BSE}$- algebra. Then $A$ is a $\textup{BSE}$- algebra.  
But $A$ is unital, so $A$ is a $\textup{BED}$-  algebra.
\end{rem}
%^^^^^^^^^^^^^^^^^^^^^^^^^^^^^^^^^^^^^^^^^^^^^^^^^^^^^^^^^^^^^^^^^^^^^^^^^^^^^^

Let $A$ be a  commutative Banach algebra with the identity of norm one.   Then $M( L^{1}(G, A))= M(G, A)$; see \cite{TE}. Also $A$ is a $\textup{BSE}$ [resp. $\textup{BED}$] algebra if and only if $(A, \| |.| \|)$, where $\| |a| \|= \frac{1}{\|e_{A}\|}. \|a\|_{A}$.
Hence without of generality, we can assume that $\|e_{A}\|=1$ in the rest of the paper.
%^^^^^^^^^^^^^^^^^^^^^^^^^^^^^^^^^^^^^^^^^^^^^^^^^^^^^^^^^^^^^^^^^^^^^^^^^^^^^^^^^^^^^^^^^^^^^^^^^^^^^^^^^^^^^^^^^^^^^
\begin{thm}\label{tlbed}
Let $A$ be commutative separable semisimple unital dual Banach algebra. Then  $ L^{1}(G, A)$ is a $\textup{BED}$- algebra  if and only if $A$ is a $\textup{BED}$- algebra.
\end{thm}
\begin{proof}
Let $A$ be a $\textup{BED}$- algebra and $\sigma\in C_{\textup{BSE}}^{0}(\Delta( L^{1}(G, A)))$. Then due to the \cite{AR} we have
$$
\sigma\in C_{\textup{BSE}}^{0}(\Delta( L^{1}(G, A)))\subseteq C_{\textup{BSE}}(\Delta( L^{1}(G, A)))= {M(G, A\hat)}
$$ 
Since $A$ is unital and $\textup{BED}$- algebra, so $A$ is a $\textup{BSE}$- algebra.Then $ L^{1}(G, A)$ is a $\textup{BSE}$- algebra and moreover 
$$M(G, A\hat)= {\cal M}(L^{1}(G, A)).$$
 Therefore there exists some $\mu\in M(G, A)$ where $\sigma= \hat \mu$.
Since 
$$M(G, A)= M(G)\hat\otimes A$$
 there exists some sequence $(\mu_{n})$ in $M(G)$ and $(a_{n})$ in $A$ such that 
$$\mu = \sum_{n=1}^{\infty} \mu_{n}\otimes a_{n}$$
 and 
$$\sum_{n=1}^{\infty} \|\mu_{n}\|.\|a_{n}\|< \infty.$$
By applying Lemma \ref{l2.2} $\mu\in L^{1}(G, A)$. Thus 
$$
C_{\textup{BSE}}^{0}(\Delta( L^{1}(G, A)))\subseteq {L^{1}(G, A \widehat)}
$$
Conversely, if $f\in L^{1}(G, A)$, then there exist some sequences $(f_{n})$ in $L^{1}(G)$ and  $(a_{n})$ in $A$  where
 $$
f= \sum_{n=1}^{\infty} f_{n}\otimes a_{n}
$$ 
with 
$$\sum_{n=1}^{\infty}\| f_{n}\| \|a_{n}\|<\infty.$$ 
Therefore $f= \lim_{n}g_{n}$ where
\begin{align*}
 \widehat{g_{n}} &= \sum_{k=1}^{n}{(f_{k}\otimes a_{k}\widehat)}\\
                      &=  \sum_{k=1}^{n}{\hat{f_{k}}\otimes \widehat{a_{k}}}\in C_{\textup{BSE}}^{0}(\Delta( L^{1}(G, A)))
\end{align*}
 at a result $\hat{f}\in C_{\textup{BSE}}^{0}(\Delta( L^{1}(G, A)))$. In fact, since  $g_{n}\to f$, so $\hat{g_{n}}\to \hat{f}$ and
$C_{\textup{BSE}}^{0}(\Delta( L^{1}(G, A)))$ is a closed ideal of $C_{\textup{BSE}}(\Delta( L^{1}(G, A)))$, so $\hat {f}\in C_{\textup{BSE}}^{0}(\Delta( L^{1}(G, A)))$. Consequently
$$
{ L^{1}(G, A\hat)}\subseteq C_{\textup{BSE}}^{0}(\Delta( L^{1}(G, A)))
$$  
Therefore $ L^{1}(G, A)$ is a $\textup{BED}$- algebra.

   Now, let $ L^{1}(G, A)$ be a $\textup{BED}$- algebra. Then by applying Lemma \ref{lbed}, $A$ is a $\textup{BED}$- algebra.

\end{proof}
%^^^^^^^^^^^^^^^^^^^^^^^^^^^^^^^^^^^^^^^^^^^^^^^^^^^^^^^^^^^^^^^

%^^^^^^^^^^^^^^^^^^^^^^^^^^^^^^^^^^^^^^^^^^^^^^^^^^^^^^^^^^^^^^^^^^^^^^^^^^^^^^^^^^^^^^^
\begin{lem}\label{lwbe}
Let  $A$ be a commutative semisimple unital Banach algebra,  $G$ be a locally compact Abelian group, and  $L^{1}(G, A)$ is a weak-$\textup{BED}$ algebra, then $A$ is a  weak-$\textup{BED}$ algebra and $G$ is discrete.
\end{lem}
\begin{proof}
If
\begin{align*}
C_{\textup{BSE}}^{0}(\Delta(L^{1}(G, A)))= {M(L^{1}(G, A)\widehat)}= {M(G, A)}
\end{align*}
Now, we show that the 
$$C_{\textup{BSE}}^{0}(\Delta(L^{1}(G)))={M(G \widehat)}.$$
 Clearly, 
$$
 C_{\textup{BSE}}^{0}(\Delta(L^{1}(G)))\subseteq C_{\textup{BSE}}(\Delta(L^{1}(G)))= M(G\widehat)
$$
See \cite{WR}.
  For the reverse inclusion, take $\mu\in M(G)$. That $\hat{\mu}\in C_{\textup{BSE}}^{0}(\Delta(L^{1}(G)))$ is assessed. If $a\in A$, then 
$$\mu\otimes a\in M(G)\widehat\otimes A\subseteq M(G, A).$$
 This implies that 
$$
{(\mu\otimes a\widehat)}\in C_{\textup{BSE}}^{0}(\Delta(L^{1}(G, A))
$$
So 
  for all $\epsilon> 0$, there exists some compact set $K\subseteq \hat{G}\otimes \Delta(A)$ where for all $c_{i}\in\mathbb C$ and $\psi_{i}\in \hat{G}\otimes \Delta(A)\backslash K$, the following inequality
$$
 |\sum_{i=1}^{n} c_{i}\widehat{\mu\otimes a}(\chi_{i}\otimes \varphi_{i})|\leq \epsilon\|\sum_{i=1}^{n} c_{i}(\chi_{i}\otimes \varphi_{i})\|_{{L^{1}(G,A)}^{*}}
$$
holds. Define 
$$\Theta_{1}: \hat{G}\otimes\Delta(A)\to \hat{G}$$
 where 
$$\chi\otimes\varphi\mapsto \chi$$
 Thus $\Theta_{1}$ is continuous. So $F:= \Theta_{1}(K)$ is compact  in $\hat G$. If $\varphi\in\Delta(A)$ where $\varphi(a)=1$ and $\varphi_{i}=\varphi$, ($1\leq i\leq n$), the following is yield:
\begin{align*}
|\sum_{i=1}^{n} c_{i}\hat{\mu}(\chi_{i})| &= |\sum_{i=1}^{n} c_{i}{(\mu\otimes a\widehat)}(\chi_{i}\otimes\varphi_{i})|\\
                                                                       &\leq \epsilon\|\sum_{i=1}^{n} c_{i}(\chi_{i}\otimes\varphi)\|_{{L^{1}(G, A)}^{*}}\\
                                                                        &\leq \epsilon\|\sum_{i=1}^{n} c_{i}\chi_{i}\|_{{L^{1}(G)}^{*}}
\end{align*}
Therefore 
$$\hat \mu\in C_{\textup{BSE}}^{0}(\Delta(L^{1}(G)))={L^{1}(G \widehat)}.$$
 Which yield:
$$
{M(G\widehat)}\subseteq {L^{1}(G\widehat)}
$$ 
so ${M(G\widehat)}= {L^{1}(G\widehat)}$, thus $M(G)= L^{1}(G)$ and $G$ is discrete. Next, we will prove that 
$$C_{\textup{BSE}}^{0}(\Delta(A))= {M(A\widehat)}= \hat A.$$ 
Since $G$ is discrete, so by applying Lemma \ref{gdis}, $L^{1}(G, A)= M(G, A)$. Thus  $L^{1}(G, A)$ is $\textup{BED}$- algebra. Hence by applying Theorem \ref{tlbed}, $A$ is $\textup{BED}$-  algebra. Therefore
$$
C_{\textup{BSE}}^{0}(\Delta(A)) = \widehat A= M(A\widehat)
$$
So $A$ is  weak-$\textup{BED}$ algebra.
\end{proof}
%^^^^^^^^^^^^^^^^^^^^^^^^^^^^^^^^^^^^^^^^^^^^^^^^^^^^^^^^^^^^^^^^^^^^^^^^^^^^^^^^^^^^^^^
\begin{thm}
Let $G$ be an Abelian locally compact group and $A$ be a commutative semisimple unital  Banach algebra. Then $L^{1}(G, A)$ is a   weak-$\textup{BSE}$ algebra if and only if $A$ is a   weak-$\textup{BSE}$ algebra and $G$ is discrete.
\end{thm}
\begin{proof}
Assume that $A$  is a   weak-$\textup{BSE}$ algebra and $G$ is discrete. Since $A$ is unital, then $A$  is a   $\textup{BSE}$- algebra, so due to \cite{AR} is a $\textup{BSE}$ algebra. Thus by using Lemma \ref{lbed}
\begin{align*}
C_{\textup{BSE}}(\Delta(L^{1}(G, A))) = {M(G, A\widehat)}={L^{1}(G, A \widehat)}
\end{align*}
as  a result $L^{1}(G, A)$ is a  weak- $\textup{BED}$ algebra.

   Conversely, if $L^{1}(G, A)$ is a  weak- $\textup{BED}$ algebra, then
$$
C_{\textup{BSE}}(\Delta(L^{1}(G, A))) = {L^{1}(G, A \widehat)}
$$
Since  $L^{1}(G, A)$ has bounded approximate identity, by Lemma \ref{Lsv} 
$$
M(G, A\widehat)= M(L^{1}(G, A)\widehat)\subseteq C_{\textup{BSE}}(\Delta(L^{1}(G, A)))
$$
Hence $M(G, A)= L^{1}(G, A)$, so by Lemma \ref{gdis} $G$ is discrete. Thus $ L^{1}(G, A)$ is a $\textup{BSE}$- algebra.
According to \cite{AR}, $A$ is $\textup{BSE}$- algebra. But $A$ is semisimple and unital, so 
$$
C_{\textup{BSE}}(\Delta(A))= \widehat{M(A)}= \widehat A.
$$
So $A$ is weak-$\textup{BSE}$ algebra.
\end{proof}
\begin{ex}
(i) If $G$ is discrete, then:
$$
M(G, A)=l_{1}(G, A)=  L^{1}(G, A)
$$
Thus $ L^{1}(G, A)$ is a $\textup{BSE}$- algebra [resp. $\textup{BED}$- algebra] if and only if $ L^{1}(G, A)$ is a weak- $\textup{BSE}$ algebra [resp. weak- $\textup{BED}$ algebra].\\
(ii) If $X$ is compact, then $C_{0}(X, A)= C_{b}(X, A)$. Thus $C_{0}(X, A)$ is $\textup{BSE}$- algebra[resp. $\textup{BED}$- algebra] if and only if  $C_{0}(X, A)$ is weak- $\textup{BSE}$ algebra[resp. weak- $\textup{BED}$ algebra].
\end{ex}

%^^^^^^^^^^^^^^^^^^^^^^^^^^^^^^^^^^^^^^^^^^^^^^^^^^^^^^^^^^^^^^^^^^^^^^^^^^^^^^^^^^^^^^^^^

\end{document}